\documentclass{amsart}
\usepackage{amssymb}
\usepackage{pb-diagram}
\usepackage{color}
\usepackage[normalem]{ulem}

\newtheorem{theorem}{Theorem}[section]
\newtheorem{thm}[theorem]{Theorem}
\newtheorem{prop}[theorem]{Proposition}

\newtheorem{fact}[theorem]{Fact}
\newtheorem{cor}[theorem]{Corollary}

\newtheorem{lemma}[theorem]{Lemma}

\newtheorem{question}[theorem]{Question}

\theoremstyle{definition}
\newtheorem{defn}[theorem]{Definition}

\theoremstyle{remark}

\newtheorem{remark}[theorem]{Remark}

\newcommand{\WM}{\widetilde{\cal M}}

\newcommand{\WR}{\widetilde{\cal R}}

\newcommand{\la}{\langle}
\newcommand{\ra}{\rangle}

\newcommand{\sub}{\subseteq}
\newcommand{\elsub}{\preccurlyeq}

\newcommand{\dcl}{\operatorname{dcl}}

\newcommand{\cl}{\operatorname{cl}}

\newcommand{\ldim}{\ensuremath{\textup{ldim}}}

\newcommand{\bb}[1]{\ensuremath{\mathbb{#1}}}

\newcommand{\cal}[1]{\ensuremath{\mathcal{#1}}}
\newcommand{\Cal}[1]{\ensuremath{\mathcal{#1}}}

\newcommand{\res}{\ensuremath{\upharpoonright}}

\newcommand{\es}{\ensuremath{\emptyset}}

\newcommand{\sm}{\setminus}

\newcommand{\Z}{\mathbb{Z}}

\newcommand{\Q}{\mathbb{Q}}
\newcommand{\R}{\mathbb{R}}

\title[Counting algebraic points in expansions of o-minimal structures]{Counting algebraic points in expansions of o-minimal structures by a dense  set}


\date{\today}
\begin{document}

\author {Pantelis  E. Eleftheriou}

\address{Department of Mathematics and Statistics, University of Konstanz, Box 216, 78457 Konstanz, Germany}

\email{panteleimon.eleftheriou@uni-konstanz.de}

\thanks{Research supported by a Research Grant from the German Research Foundation (DFG) and a Zukunftskolleg Research Fellowship.}

\begin{abstract} Let $\cal R =\la \R, <, +, \cdot, \dots\ra$ be an o-minimal expansion of the real field, and $P\sub \R$ a dense elementary substructure of $\cal R$ or a dense
independent set.  We prove that if  $X\sub \R^n$ is definable in the pair $\la \cal R,  P\ra$ and contains ``many" algebraic points, then it is dense in an infinite semialgebraic set.
\end{abstract}

\begin{abstract}
The Pila-Wilkie theorem states that if a set $X\sub \R^n$ is definable in an o-minimal structure \cal R  and  contains `many' rational points, then it contains an infinite semialgebraic set.
In this paper, we extend this theorem to an expansion $\WR=\la \cal R, P\ra$ of $\cal R$ by a dense set $P$, which is either an elementary substructure of $\cal R$, or it is independent, as follows. If $X$ is definable in $\WR$ and contains many rational points, then it is dense in an infinite semialgebraic set. Moreover, it contains an infinite set which is $\es$-definable in $\la \overline \R, P\ra$, where $\overline \R$ is the real field.
\end{abstract}

\maketitle

\section{Introduction}

Point counting theorems have recently occupied an important part of model theory, mainly due to their pivotal role in applications of o-minimality to number theory and Diophantine geometry. Arguably, the biggest breakthrough was the Pila-Wilkie theorem \cite{pw}, which roughly states that  if a definable set in an o-minimal structure contains ``many'' rational points, then it contains an infinite semialgebraic set. Pila employed this result together with the so-called Pila-Zannier strategy to give an unconditional proof of certain cases of the Andr\'e-Oort Conjecture \cite{pila2}. An excellent survey on the subject is \cite{scanlon}. Although several strengthenings of these theorems have since been established within the o-minimal setting, the topic remains largely unexplored in more general tame settings.  In this paper, we establish the first point counting theorems in tame expansions of o-minimal structures by a dense set.

Recall that, for a set $X\sub \R^n$, the \emph{algebraic part} $X^{alg}$ of $X$ is defined as the union of all infinite connected semialgebraic subsets of $X$. Pila in \cite{pila2}, generalizing \cite{pw}, proved that if a set $X$ is definable in an o-minimal structure, then $X\sm X^{alg}$ contains ``few" algebraic points of  fixed degree (see  definitions below and Fact \ref{pw}). This statement immediately fails if one leaves the o-minimal setting. For example, the set $\cal A$ of algebraic points itself contains many algebraic points, but $\cal A^{alg}=\es$. However, adding $\cal A$ as a unary predicate to the language of the real field results in a well-behaved model theoretic structure, and it is desirable to retain point counting theorems in that setting. We achieve this goal by means of the following definition.



\begin{defn}\label{def-alg2} Let $X\sub \R^n$. The \emph{algebraic trace part} of $X$, denoted by $X^{alg}_{t}$, is the union of all traces of infinite connected semialgebraic sets in which $X$ is dense. That is,
$$X^{alg}_t=\bigcup\{X\cap T: \text{$T\sub \R^n$  infinite connected semialgebraic, and $T\sub cl(X\cap T)$}\}.$$
\end{defn}

 \noindent The density requirement $T\sub cl(X\cap T)$ is essential: without it, we would always have $X_t^{alg}=X$, as witnessed by $T=\R^n$.

We first show in Section \ref{sec-alg} that the above notion is a natural generalization of the usual notion of the algebraic part of a set, in the following sense.


\begin{prop}
Suppose $X\sub \R^n$ is definable in an o-minimal expansion of the real field. Then $X^{alg}=X^{alg}_t$.
\end{prop}

\noindent Then, in Sections \ref{sec-dense} and \ref{sec-indep}, we establish point counting theorems in two main categories of tame structures that go beyond the o-minimal setting: dense pairs and expansions of o-minimal structures by a dense independent set. Indeed, we prove that if $X$ is a definable set in these settings, then $X\sm X^{alg}_t$ contains few algebraic points of fixed degree (Theorem \ref{main} below). We postpone a discussion about the general tame setting until later in this introduction, as we now proceed to fix our notation and state the precise theorem. Some familiarity with the basic notions of model theory, such as definability and elementary substructures, is assumed. The reader can consult \cite{vdd-book, marker, pila}. An example of an elementary substructure of the real field is the field $\cal A$ of algebraic numbers.

For the rest of this paper, and unless stated otherwise, we fix an o-minimal expansion $\cal R=\la  \R, <, +, \cdot, \dots\ra$ of the real field $\overline \R=\la \R, <, +, \cdot\ra$, and let $\cal L$ be the language of $\cal R$. We fix an expansion $\widetilde{\cal R}=\la \cal R, P\ra$  of $\cal R$ by a set $P\sub \R$, and let  $\cal L(P)=\cal L\cup \{P\}$ be the language of $\WR$. By `$A$-definable' we mean `definable in $\WR$ with parameters from $A$', and by `$\cal L_A$-definable' we mean `definable in \cal R with parameters from $A$'. We omit the index $A$ if we do not want to specify the parameters. For a subset $A\subseteq \R$, we write $\dcl(A)$ for the definable closure of $A$ in $\Cal R$, and $\dcl_{\cal L(P)}(A)$ for the definable closure in $\WR$. We call a set $X\sub \R$ \emph{$\dcl$-independent over $A$}, if for every $x\in X$, $x\not\in \dcl((X\sm\{x\}) \cup A)$, and simply \emph{$\dcl$-independent} if it is $\dcl$-independent over $\es$. An example of a $\dcl$-independent set in the real field is a transcendence basis over $\Q$.

Following \cite{pila}, we define the \emph{(multiplicative) height} $H(\alpha)$ of an algebraic point $\alpha$ as $H(\alpha)=\exp h(\alpha),$
where $h(\alpha)$ is the \emph{absolute logarithmic height} from \cite[page 16]{bg}.
For a set $X\sub \R^n$, $k\in \Z^{>0}$ and $T\in\R^{>1}$, we define
$$X(k, T)=\{(\alpha_1, \dots, \alpha_n)\in X: \max_i[\Q(\alpha_i):\Q]\le k, \max_i H(\alpha_i)\le T\}$$
and
$$N_k(X, T)=\# X(k, T).$$
We say that \emph{$X$ has few algebraic points} if for every  $k\in \Z^{>0}$ and $\epsilon\in \R^{>0}$,
$$N_k(X, T) = O_{X, k, \epsilon}(T^{\epsilon}).$$
We say that it  \emph{has many algebraic points}, otherwise.


Our main result is the following.

\begin{theorem}\label{main} Suppose $\cal R=\la \R, <, +, \cdot, \dots\ra$ is an o-minimal expansion of the real field, and $P\sub R$ a dense set such that one of the following two conditions holds:
\begin{enumerate}
  \item[(A)] $P\elsub \cal R$ is an elementary substructure.
  \item[(B)] $P$ is a $\dcl$-independent set.
\end{enumerate}
Let $X\sub \R^n$ be definable in $\WR=\la \cal R, P\ra$. Then $X\sm X^{alg}_t$ has few algebraic points.
\end{theorem}

Note that if   $\cal R=\overline \R$,  Theorem \ref{main} is trivial. Indeed, in both cases (A) and (B), if $X$ is a definable set, then $cl(X)$ is $\cal L$-definable (\cite[Section 2]{egh}). So, in this case, $cl(X)$ is semialgebraic and hence $X^{alg}_t=X$. In fact, whenever  $\WR =\la \overline \R, P\ra$ satisfies Assumption  III from \cite{egh},  the conclusion of Theorem \ref{main} holds. An example of such $\WR$ is an expansion of the real field by a multiplicative group with the Mann property.\\


The contrapositive of Theorem \ref{main} implies that if a definable set  contains many algebraic points, then it is dense in an infinite semialgebraic set. We strengthen this result as follows.




\begin{thm}\label{thm-cor}
Let $X$ be as in Theorem \ref{main}. If $X$ has many algebraic points, then it contains an infinite set $Y$ which is $\es$-definable in $\la \overline \R, P\ra$.
\end{thm}
Note that such $X$ is dense in $cl(Y)$, which is semialgebraic by \cite[Section 2]{egh}. 


\smallskip

A few words about the general tame setting 
are in order. As o-minimality can only be used to model phenomena that are locally finite, many authors  have early on sought expansions of o-minimal structures which escape from the o-minimal context,  yet preserve the tame geometric behavior on the class of all definable sets. These expansions have recently  seen significant growth
(\cite{bz, beg, bh,dms1,vdd-dense,dg,gh, ms}) and are by now divided into two important categories of structures: those where every open definable set is already definable in the o-minimal reduct and those where an infinite discrete set is definable. Cases (A) and (B) from Theorem \ref{main} belong to the first category. Further examples of this sort can be found in \cite{dms1} and \cite{egh}. Certain point counting theorems in the second category have recently appeared in \cite{cm}. In both categories, sharp \emph{cone decomposition theorems} are by now at our disposal (\cite{egh} and \cite{tycho}), in analogy with the cell decomposition theorem known for o-minimal structures.
%

Expansions  $\cal R$ of type (A) are called \emph{dense pairs} and were first studied by van den Dries in \cite{vdd-dense}, whereas expansions of type (B) were recently introduced by Dolich-Miller-Steinhorn in \cite{dms2}.  These two examples are representative of the first category and  are often thought of as ``orthogonal" to each other, mainly because in the former case $\dcl(\es)\sub P$, whereas in the latter, $\dcl(\es)\cap P=\es$. This orthogonality is vividly reflected in our proof of Theorem \ref{main}. Indeed, since the set $\cal A$ of algebraic points is contained in $\dcl(\es)$, we have $\cal A\sub P$ in the case of dense pairs and  $\cal A\cap P=\es$ in the case of dense independent sets. Based on this observation, the proof for (A) becomes almost immediate, assuming facts from \cite{vdd-dense}, whereas the proof for (B) makes an essential use of the aforementioned cone decomposition theorem from \cite{egh}. 

The current work provides an extension of the influential Pila-Wilkie theorem to the above two settings. The next step is, of course, to explore any potential applications to number theory and Diophantine geometry.  Even though it is currently unclear whether the exact setting of Theorem \ref{main} will yield any, the machinery used in our proofs is also available in other settings, or it may be possible to develop therein. 
Two far reaching generalizations of our two settings are lovely pairs \cite{bv1} and $H$-structures \cite{bv2}, respectively. Those settings can also accommodate structures coming from geometric stability theory, such as pairs of algebraically closed fields, or  $SU$-rank 1 structures, and point counting theorems in them are wildly unknown. 









\vskip.2cm
\noindent\textbf{Notation.} The topological closure of a set $X\sub \R^n$ is denoted by $cl(X).$ If $X, Z\sub \R^n$, we call $X$ dense in $Z$, if $Z\sub cl(X\cap Z)$.
Given any subset $X\subseteq \R^m \times \R^n$ and $a\in \R^m$, we write $X_a$ for
\[
\{ b \in \R^n \ : \ (a, b) \in X\}.
\]
If $m\le n$, then $\pi_m:\R^n\to \R^m$ denotes the projection onto the first $m$ coordinates. We write $\pi$ for $\pi_{n-1}$, unless stated otherwise.
A family $\cal J=\{J_g\}_{g\in S}$ of sets is called definable if $\bigcup_{g\in S}\{g\}\times J_g$ is definable.
We often identify $\cal J$ with $\bigcup_{g\in S}\{g\}\times J_g$.
If $X, Y \subseteq \R$, we sometimes write $XY$ for $X\cup Y$.
By $\cal A$ we denote the set of real algebraic points. If $M\sub \R$, by $M \elsub \cal R$ we mean that $M$ is an elementary substructure of $\cal R$ in the language of $\cal R$.

$ $\\
\noindent\textbf{Acknowledgments.} The author wishes to thank Gal Binyamini, Chris Miller, Ya'acov Peterzil, Jonathan Pila, Patrick Speissegger, Pierre Villemot and Alex Wilkie for several discussions on the topic, and the Fields Institute for its generous support and hospitality during the Thematic Program on Unlikely Intersections, Heights, and Efficient Congruencing, 2017.

\section{The algebraic trace part of a  set}\label{sec-alg}

In this section, we introduce the notion of the \emph{algebraic trace part} of a set, and prove that it  generalizes the notion of the algebraic part of a set definable in an o-minimal structure. We also state a  version of Pila's theorem \cite{pila}, Fact \ref{pw} below, suitable for our purposes. 

The proof of Theorem \ref{main}, in both cases (A) and (B), is by reducing it to Pila's theorem, Fact \ref{pw} below. 
The formulation of that fact involves a refined version of the usual algebraic part of a set, which prompts the following definitions.



\begin{defn} Let $A\sub \R$ be a set. An \emph{$A$-set} is an infinite connected semialgebraic set definable over $A$. If it is, in addition, a cell, we call it an \emph{$A$-cell}.
\end{defn}

We are mainly interested in $\Q$-sets. One important observation is that the set $\cal A$ of algebraic points is dense in every $\Q$-set. This fact will be crucial in the proofs of Lemma \ref{dense1} and Theorem \ref{main-indep} below.


\begin{defn}\label{def-alg1} Let $X\sub \R^n$ and $A\sub \R$. The \emph{algebraic part of $X$ over $A$}, denoted by $X^{alg_A}$, is the union of all $A$-subsets of $X$. That is,
  $$X^{alg_A}=\bigcup\{T\sub X : T \text{ is an $A$-set}\}.$$
\end{defn}

It is an effect of the proof in \cite{pila} that the following statement holds.
\begin{fact}\label{pw}
Let $X\sub \R^n$ be $\cal L$-definable. Then $X\sm X^{alg_\Q}$ has few algebraic points.
\end{fact}

Let us now also refine Definition \ref{def-alg2} from the introduction, as follows.


\begin{defn}\label{def-alg2b} Let $X\sub \R^n$ and $A\sub \R$. The \emph{algebraic trace part of $X$ over $A$}, denoted by $X^{alg_A}_t$ is the union of all traces of $A$-sets in which $X$ is dense. That is,
  $$X^{alg_A}_t=\bigcup\{X\cap T: \text{$T$  an $A$-set, $X$  dense in $T$}\}$$
\end{defn}

\begin{remark}$ $

(1) An $\R$-set is exactly an infinite connected semialgebraic set. Also, $X^{alg_\R}=X^{alg}$ and $X^{alg_\R}_t=X^{alg}_t$.

(2) In Theorems \ref{main-dense} and \ref{main-indep} below, we prove Theorem \ref{main} after replacing $X^{alg}_t$ by $X^{alg_\Q}_t$. Since the latter set is contained in the former, these are stronger statements.


\end{remark}

\begin{remark}\label{alternative}
An alternative expression for $X^{alg_A}_t$ is the following:
$$X^{alg_A}_t=\bigcup\{Y\sub X: \text{$cl(Y)$ is  an $A$-set}\}.$$
$\sub$. Let $T$ be an $A$-set such that $X$ is dense in $T$. Set $Y=X\cap T\sub X$. Then $T\sub cl(Y)\sub cl(T)$, and hence $cl(Y)=cl(T)$ is an $A$-set, as required.\smallskip

\noindent $\supseteq$. Let $Y\sub X$ such that $cl(Y)$ is an $A$-set. Set $T=cl(Y)$. Then $Y\sub X\cap T$ and $T\sub cl(X\cap T)$, as required.
\end{remark}



The goal of this section is to prove the following proposition.
This result is not essential for the rest of the paper, but we include it  here as it provides  canonicity of our definitions. Observe also
that it is independent of the expansion $\WR$ of $\cal R$ we consider.


\begin{prop}\label{alg12} Let $X\sub \R^n$ be an $\cal L$-definable set. Then
$$X^{alg}=X^{alg}_t.$$
\end{prop}

The main idea for proving $(\supseteq)$ is as follows.
Let $Z$ be an $\R$-set with $Z\sub cl(Z\cap X)$. We need to prove that every point $x\in Z\cap X$ is contained in an $\R$-set $W$ contained in $X$. If one  applies  cell decomposition directly to $Z\cap X$, then the resulting cells need not be semialgebraic, as $X$ is not. So we apply cell decomposition only to $Z$, deriving an $\R$-cell $Z_0\sub Z$ with $x\in cl(Z_0)$ and of maximal dimension. We then  show that close enough to $x$, the set $T=Z_0\sm X$ has dimension strictly smaller than $\dim Z_0$. We  use Lemma \ref{BZ} to express this fact properly. Finally, by Lemma \ref{clZT2}, we find an $\R$-set $W_0\sub Z_0\sm T$ with $x\in \cl(W_0)$. We set $W=W_0\cup \{x\}$.


The first lemma asserts that, under certain assumptions, the property of being dense in a set passes to  suitable subsets.

\begin{lemma}\label{denseZ0} Let $X, Z\sub \R^n$ be $\cal L$-definable sets, with $Z\sub cl(Z\cap X)$. Suppose that $Z_0\sub Z$ is a cell with $\dim Z_0 =\dim Z$. Then $Z_0\sub cl(Z_0\cap X)$.
\end{lemma}
\begin{proof}
 Let $x\in Z_0$, and suppose towards a contradiction that $x\not\in cl(Z_0\cap X)$. Then there is an open box $B\sub \R^n$ containing $x$ such that
$B\cap Z_0\cap X=\es$. It follows  that for every $x'\in B\cap Z_0$, $x'\not\in cl(Z_0\cap X)$. 
Since $Z\sub cl(Z\cap X)$,
$$B\cap Z_0\sub cl((Z\sm Z_0)\cap X)\sub cl(Z\sm Z_0)$$
and, hence,
$$B\cap Z_0\sub cl(Z\sm Z_0)\sm (Z\sm Z_0),$$
and thus $\dim(B\cap Z_0)<\dim(Z\sm Z_0)$.
Moreover, since $Z_0$ is a cell and $B\cap Z_0\ne \es$, $\dim(Z_0)=\dim (B\cap Z_0)$. All together,
$$\dim(Z_0)<\dim (Z\sm Z_0)\le \dim Z,$$ 
a contradiction.
\end{proof}

We will need a local version of Lemma \ref{denseZ0}. First, a definition.

\begin{defn}
Let $Z\sub \R^n$ be an $\cal L$-definable set and $x\in Z$. The \emph{local dimension of $Z$ at $x$} is defined to be
$$\dim_x(Z)=\min \{\dim (B\cap Z):\text{ $B\sub \R^n$ an open box containing $x$}\}.$$
\end{defn}

\begin{lemma}\label{BZ} Let $X, Z\sub \R^n$ be infinite $\cal L$-definable sets with $Z\sub cl(Z\cap X)$, and $x\in Z$. Suppose $Z_0\sub Z$ is an $\R$-cell with $\dim_x(Z)=\dim Z_0$ and $x\in cl(Z_0)$. Then there is an open box $B\sub \R^n$ containing $x$, such that  $B\cap Z_0\sub cl(Z_0\cap X)$. Moreover, $B\cap Z_0$ is an $\R$-cell.
\end{lemma}
\begin{proof}
Let $Z\sm Z_0=Z_1\cup \dots  \cup Z_m$ be a decomposition into cells. It is not hard to see from the definition of $\dim_x(Z)$, that there is an open box $B\sub \R^n$ containing $x$, such that for every $1 \le i\le m$, if $B\cap Z_i\ne \es$, then $\dim_x(Z) \ge \dim B\cap Z_i$. We may shrink $B$ if needed so that $B\cap Z_0$ becomes an $\R$-cell. Let $I$ be the set of indices $1\le i\le m$ such that $B\cap Z_i\ne \es$. Set
$$Z':=B\cap Z.$$
Since $Z\sub cl(Z\cap X)$, we easily obtain that $Z'\sub cl(Z'\cap X)$. Moreover, since $x\in cl(Z)$, we have
$$Z' =(B\cap Z_0)\cup \bigcup_{i\in I} (B\cap Z_i),$$
and hence $\dim Z'=\dim (B\cap Z_0)$. Therefore, by Lemma \ref{denseZ0} (for $Z'$ and $B\cap Z_0\sub Z'$),
$$B\cap Z_0\sub cl(B\cap Z_0\cap X)\sub cl(Z_0\cap X),$$
as needed.
\end{proof}

We also need the following lemma.

\begin{lemma}\label{clZT2}
Let $Z\sub \R^n$ be an $\R$-cell, $T\sub Z$ a definable set, and $x\in cl(Z)\sm T$. Suppose that $\dim T<\dim Z$.
Then there is an $\R$-set $W\sub Z\sm T$ with $x\in cl(W)$.
\end{lemma}


\begin{proof} We work by induction on $n>0$. For $n=0$, it is trivial. 
Let $n>0$. We split into two cases:\smallskip

\noindent Case I: $\dim Z=n$. Since $\dim T< \dim Z$, it follows easily, by cell decomposition, that there is a line segment $W\sub Z$  with initial point $x$, staying entirely outside $T$.\smallskip

\noindent Case II: $\dim Z=k< n$. Let $\pi:\R^n\to \R^k$ be a suitable coordinate projection such that $\pi_{\res Z}$ is injective. Then $\pi(Z)$ is an $\R$-cell, $\pi(T)\sub \pi(Z)$, $\dim \pi(T)< \dim \pi(Z)$ and $\pi(x)\in cl(\pi(Z))$. By inductive hypothesis, there is an $\R$-set $W_1\sub \pi(Z)\sm \pi(T)$, such that $\pi(x)\in cl(W_1)$. Let
$$W=\pi^{-1}(W_1)\cap Z.$$
Then $W$ is clearly an $\R$-set with  $W\sub Z\sm T$, and it is also easy to check that $x\in cl(W)$.
\end{proof}

We are now ready to prove Proposition \ref{alg12}.

\begin{proof}[Proof of Proposition \ref{alg12}] We need to show $X^{alg}_t\sub X^{alg}$. Let $Z$ be an $\R$-set with $Z\sub cl(Z\cap X)$. We need to prove that every point $x\in Z\cap X$ is contained in an $\R$-set $W$ contained in $X$. By cell decomposition in the real field, 
there is a semialgebraic cell $Z_0\sub Z$ over $A$,
such that $\dim_x(Z)=\dim Z_0$ and $x\in cl(Z_0)$. By Lemma \ref{BZ}, there is an open box $B\sub \R^n$ containing $x$, such that $B\cap Z_0$ is an $\R$-cell and $B\cap Z_0\sub cl(Z_0\cap X)$. Let
$$T=(B\cap Z_0)\sm (Z_0\cap X)\sub cl(Z_0\cap X)\sm (Z_0\cap X).$$
Then
$$\dim T < \dim (Z_0\cap X)\le \dim Z_0=\dim (B\cap Z_0).$$ 
Also, $x\in Z\sm T$. Therefore, by Lemma \ref{clZT2} (for $Z=B\cap Z_0$), 
 there is an $\R$-set $W_0\sub (B\cap Z_0)\sm T$ with $x\in cl(W_0)$. But
$$(B\cap Z_0) \sm T=B\cap Z_0\cap X,$$
so $W_0\sub X$. Since $x\in cl(W_0)$, the set  $W=W_0\cup \{x\}$ is connected, and hence the desired $\R$-set.
\end{proof}


\begin{remark}\label{XalgQ} If we specify parameters in Proposition \ref{alg12}, then the proposition need not be true. 
Indeed
$$X^{alg_\Q}\ne X^{alg_\Q}_t.$$ For example, fix a $\dcl$-independent tuple $a=(a_1, a_2)\in \R^2$, and let
$$X=\R^2\sm \{(a_1, y): y>a_2\}.$$
Then $a\in X\sub X^{alg_\Q}_t$, since $cl(X)=\R^2$ is a $\Q$-set.
However, $a\not\in X^{alg_\Q}$.
 Indeed, no open box around $a$ can be contained in $X$. Hence if $a\in X^{alg_\Q}$, there must be some $1$-dimensional semialgebraic set over $\es$ that contains $a$, contradicting the $\dcl$-independence of $a$.
Note that in the proof of Proposition \ref{alg12},  unless $x\in \dcl(\es)$, we cannot conclude that $W$ is semialgebraic over $\es$.

We do not know whether $X^{alg_A}=X^{alg_A}_t$ is true if $X$ is  $A$-definable.
\end{remark}



\begin{remark}
The proof of Proposition \ref{alg12} uses nothing in particular about the real field. In other words, if we fix an expansion $\WM$ of any real closed field $\cal M$, and define the notions of $X^{alg}$ and $X^{alg}_t$ in the same way as in the introduction after replacing `semialgebraic' by `\cal M-definable', and `connected' by `\cal M-definably connected', then for every \cal M-definable set $X$, we have $X^{alg}=X^{alg}_t$.
\end{remark}





We conclude this section with an easy fact.
\begin{fact}\label{union}
Let $X, Y\sub \R^n$ be two definable sets.
\begin{enumerate}
  \item If $X\sub Y$, then $X^{alg_\Q}_t\sub Y^{alg_\Q}_t$.

  \item  (a) If $X\sub Y$ and $Y$ has few algebraic points, then so does $X$.

\noindent (b) If  $X$ and $Y$ have few algebraic points, then so does $X\cup Y$.

  \item If  $X\sm X^{alg_\Q}_t$ and $Y\sm Y^{alg_\Q}_t$ have few algebraic points, then so does \linebreak $(X\cup Y) \sm (X\cup Y)^{alg_\Q}_t $.

\end{enumerate}
\end{fact}
\begin{proof}


\noindent (1) and (2) are obvious. For (3), we have:
$$(X\cup Y) \sm (X\cup Y)^{alg_\Q}_t \sub (X\sm (X\cup Y)^{alg_\Q}_t) \cup (Y\sm (X\cup Y)^{alg_\Q}_t)\sub (X\sm X^{alg_\Q}_t)\cup (Y\sm Y^{alg_\Q}_t),$$
and  we are done by (2).
\end{proof}


\section{Dense pairs}\label{sec-dense}

In this section, we let $\WR=\la \cal \R, P\ra$ be a dense pair.  As mentioned in the introduction, since $P\elsub \cal R$, we have $\cal A\sub\dcl(\es)\sub P$. 
In this setting, Theorem \ref{thm-cor} has a short and illustrative proof, and we include it first. 

\begin{thm}
For every definable set $X$, if $X$ has many algebraic points, then it contains an infinite set which is $\es$-definable in $\la \overline \R, P\ra$.
\end{thm}
\begin{proof}
Since $\cal A\sub P$, $X\cap P^n$ also contains many algebraic points. By \cite[Theorem 2]{vdd-dense}, there is an $\cal L$-definable $Y\sub \R^n$, such that $X=Y \cap P^n$. So $Y$ also contains many algebraic points. By Fact \ref{pw}, there is a $\Q$-set $Z\sub Y$. Then the set $Z\cap P^n$ is $\es$-definable in $\la \overline \R, P\ra$ and it is contained in $Y\cap P^n=X$. Since the set of algebraic points $\cal A^n$ is dense in $Z$, we have $Z\sub cl(Z\cap \cal A^n)\sub cl(Z\cap P^n)$, and hence $Z\cap P^n$ is  infinite.
\end{proof}

 We now proceed to the proof of Theorem \ref{main}.

\begin{lemma}\label{dense1}
Let $X=Y\cap P^n$, for some $\cal L$-definable set $Y\sub \R^n$. Then
$$X\cap Y^{alg_\Q}\sub X^{alg_\Q}_t.$$
\end{lemma}
\begin{proof}
Let $x\in X\cap Y^{alg_\Q}$. So $x$ is contained in a $\Q$-set $Z\sub Y$. We prove that $X$ is dense in $Z$.
Observe that $Z\cap X= Z\cap P^n$.  Since  $\cal A^n\sub P^n$, we have
$$Z\sub cl(Z\cap \cal A^n)\sub cl(Z\cap P^n)= cl(Z\cap X),$$ 
and hence $X$ is dense in $Z$.
\end{proof}

\begin{thm}\label{main-dense} For every definable set $X$, $X\sm X^{alg_\Q}_t$ has few algebraic points.
\end{thm}
\begin{proof} Let $k\in \Z^{>0}$ and $\epsilon\in \R^{>0}$. We first observe that if the statement holds for $X\cap P^n$, then it holds for $X$. Of course, $X\sm X^{alg_\Q}_t\sub X\sm (X\cap P^n)^{alg_\Q}_t$. Since $\cal A^n\sub P^n$, the set  $X$ has the same algebraic points as  $X\cap P^n$, and hence if $(X\cap P^n)\sm (X\cap P^n)^{alg_\Q}_t$ has few algebraic points, then so does $X\sm (X\cap P^n)^{alg_\Q}_t$, and therefore also $X\sm X^{alg_\Q}_t$.

We may thus assume that $X\sub P^n$. By \cite[Theorem 2]{vdd-dense}, there is an $\cal L$-definable $Y\sub \R^n$, such that $X=Y \cap P^n$. By Fact \ref{pw}, $Y\sm Y^{alg_\Q}$ has few algebraic points.  By Lemma \ref{dense1},
$$ X\cap Y^{alg_\Q}\sub X^{alg_\Q}_t.$$
Hence
$$X\sm X^{alg_\Q}_t\sub X\sm Y^{alg_\Q}\sub Y\sm Y^{alg_\Q}$$
has few  algebraic points.
\end{proof}



\section{Dense independent sets}\label{sec-indep}

In this section, $P\sub \R$ is a dense $\dcl$-independent set. The proof of Theorem \ref{main-indep} runs by induction on the \emph{large dimension} of a definable set $X$ (Definition \ref{def-large}), by making  use of the \emph{cone decomposition theorem} from \cite{egh} (Fact \ref{def-cone}). As mentioned in the introduction, since $P$ contains no elements in $\dcl(\es)$, we have $P\cap \cal A=\es$. The base step of the aforementioned induction is to show a generalization of this fact; namely, that for a \emph{small} set $X$ (Definition \ref{def-small}),  $X\cap \cal A$ is finite (Corollary \ref{smallA}). 

\subsection{Cone decomposition theorem}\label{sec-cone}

In this subsection we recall all necessary background from \cite{egh}. The following definition is taken essentially from \cite{dg}.

\begin{defn}\label{def-small}
 Let $X\sub \R^n$ be a definable set. We call $X$ \emph{large} if there is some $m$ and an $\cal L$-definable function $f:\R^{nm}\to \R$ such that $f(X^m)$ contains an open interval in $\R$. We call $X$ \emph{small} if it is not large.
\end{defn}


The notion of a cone is based on that of a supercone, which in its turn generalizes the notion of being co-small in an interval. Both supercones and cones are unions of special families of sets, which not only are definable, but they are so in a very uniform way. Although this uniformity  is not fully exploited in this paper, we include it here to match the definitions from \cite{egh}.

\begin{defn}[\cite{egh}]\label{def-supercone}
A \emph{supercone} $J\sub \R^k$, $k\ge 0$, and its \emph{shell} $sh(J)$ are defined recursively as follows:
\begin{itemize}
\item $\R^{0}=\{0\}$ is a supercone, and $sh(\R^{0})=\R^{0}$.

\item A definable set $J\sub \R^{n+1}$ is a supercone if $\pi(J)\sub \R^n$ is a supercone and there are  $\cal L$-definable continuous $h_1, h_2: sh(\pi(J))\to \R\cup \{\pm\infty\}$ with $h_1<h_2$, such that for every $a\in \pi(J)$, $J_a$ is contained in $(h_1(a), h_2(a))$ and it is co-small in it. We let $sh(J)=(h_1, h_2)_{sh(\pi(J))}$.
\end{itemize}
\end{defn}

Note that, $sh(J)$ is an open cell in $\R^k$ and $cl(sh(J))=cl(J)$.

Recall that in our notation we identify a family $\cal J=\{J_g\}_{g\in S}$ with $\bigcup_{g\in S}\{g\}\times J_g$. In particular, $cl(\cal J)$ and $\pi_n(\cal J)$ denote the closure and a projection of that set, respectively.




\begin{defn}[Uniform families of supercones \cite{egh}]\label{def-uniform}
Let $\Cal J = \bigcup_{g\in S} \{g\}\times J_g\sub \R^{m+k}$ be a definable family of supercones. We call $\Cal J$ \emph{uniform} if there is a cell $V\sub \R^{m+k}$ containing $\cal J$, such that for every $g\in S$ and  $0<j\le k$,
$$cl(\pi_{m+j}(\cal J)_g)=cl(\pi_{m+j}(V)_g).$$
We call such a $V$ a \emph{shell} for $\cal J$. 
\end{defn}


\begin{remark}
A shell for a uniform family of supercones $\cal J$ need not be unique. Also, one can identify a supercone $J\sub \R^k$ with a uniform family of supercones $\cal J\sub M^{m+k}$ with $\pi_m(\cal J)$ a singleton; in that case, a shell for $\cal J$ is unique and equals that of $J$. 
\end{remark}

\begin{defn}[Cones \cite{egh} and $H$-cones\footnote{The letter `$H$' derives from `Hamel basis' - see \cite{dms2} for the motivating example $\la \R, <, +, H\ra$.}] 
\label{def-cone}
A set $C\sub \R^n$ is a \emph{$k$-cone}, $k\ge 0$,  if there are a definable small $S\sub \R^m$, a uniform family $\cal J=\{J_g\}_{g\in S}$ of supercones in $\R^k$, and an $\cal L$-definable continuous function $h:V \subseteq \R^{m+k}\to \R^n$, where $V$ is a shell for $\cal J$, such that
\begin{enumerate}
\item $C=h(\cal J)$, and
\item for every $g\in S$, $h(g, -): V_g\sub \R^k\to \R^n$ is injective.
\end{enumerate}
We call $C$ a \emph{$k$-$H$-cone} if, in addition,  $S\sub P^m$ and $h:\cal J\to \R^n$ is injective. An \emph{($H$-)cone} is a $k$-($H$-)cone for some $k$.
\end{defn}

The cone decomposition theorem \cite[Theorem 5.1]{egh} is a statement about definable sets and functions. Here we are only interested in a decomposition of sets into $H$-cones. Before stating the \emph{$H$-cone decomposition theorem}, we  need the following fact.






\begin{fact}\label{cp} Let $S\sub \R^n$ be an $A$-definable small set. Then $S$ is a finite union of sets of the form $f(X)$, where
\begin{itemize}
  \item  $f:Z\sub \R^m\to \R^n$ is an $\cal L_A$-definable continuous map,
  \item $X\sub  P^m\cap Z$ is $A$-definable, and
\item $f: X\to \R^l$ is injective.
\end{itemize}
\end{fact}
\begin{proof} By \cite[Lemma 3.11]{egh}, there is an $\cal L_A$-definable map $h:\R^m\to \R^n$ such that $X\sub h(P^m)$. The result follows from \cite[Theorem 2.2]{egh2}. 
\end{proof}

 \begin{fact}[$H$-cone decomposition theorem]\label{fact-cone}
Let $X\sub \R^n$ be an $A$-definable set. Then $X$ is a finite union of $A$-definable $H$-cones.
 \end{fact}
\begin{proof}
By \cite[Theorem 5.12]{egh} and \cite[Theorem 2.2]{egh2}, $X$ is a finite union of $A$-definable cones $h(\cal J)$ with $h:\cal J\to \R^n$ injective (such $h(\cal J)$ is called  \emph{strong cone} in the above references). By Fact \ref{cp}, it is not hard to see that $h(\cal J)$ is a finite union of $A$-definable $H$-cones.
\end{proof}



We next recall the notion of `large dimension' from \cite{egh}.


\begin{defn}[Large dimension \cite{egh}]\label{def-large}
Let $X\sub \R^n$ be definable. If $X\ne \emptyset$, the \emph{large dimension} of $X$ is the maximum $k\in \bb N$ such that $X$ contains a $k$-cone. The large dimension of the empty set is defined to be $-\infty$. We denote the large dimension of $X$ by  $\ldim(X)$.
\end{defn}

Some basic properties of the large dimension that  will be used in the sequel are the following (see \cite[Lemma 6.11]{egh}): for every two definable sets $X, Y\sub \R^n$,
\begin{itemize}
\item if $X\sub Y$, then $\ldim X\le\ldim Y.$
\item if $X$ is $\cal L$-definable, then $\ldim X=\dim X$.
\item $X$ is small if and only if $\ldim X=0$.
\end{itemize}

\subsection{Point counting} We now proceed to the proof of Theorem \ref{main} (B). We  need several preparatory lemmas. First, a very useful fact.

\begin{fact}\label{op} For every $A\sub \R$ with $A\sm P$ $\dcl$-independent over $P$, we have $\dcl_{\cal L(P)}(A)=\dcl(A)$.
\end{fact}
\begin{proof}
Take $x\in \dcl_{\cal L(P)}(A)$. That is, the set $\{x\}$ is $A$-definable in $\la \cal R, P\ra$. By \cite[Assumption III]{egh},  since $A\sm P$ is $\dcl$-independent over $P$, we have that $cl(\{x\})$ is $\cal L_A$-definable. But $cl(\{x\})=\{x\}$. So $x\in \dcl(A)$.
\end{proof}

The following lemma is crucial and relies on the fact that $P$ is $\dcl$-independent.

\begin{lemma}\label{S0c} Let $h:Z\sub P^m\times \R^k\to \R^n$ be a definable injective map. Let $B\sub \R$ be a finite set. Then there is a finite set $S_0\sub P^m$ such that
$$h\left(\bigcup_{g\in P^m\sm S_0}\{g\}\times Z_g\right)\cap \dcl(B)^n=\es.$$
\end{lemma}

\begin{proof}
  Suppose $h$ is $A$-definable, with $A$ finite. Let $A_0\sub A\cup B$ and $P_0\sub P$ be finite  so that $A\cup B\sub \dcl(A_0P_0)$ and $A_0$ is $\dcl$-independent over $P$. Suppose $q=h(g, t)$, where $g\in P^m$,  $t\in Z_g$ and  $q\in \dcl(B)$. By injectivity of $h$, all coordinates of $g$ are in
$$\dcl_{\cal L(P)}(A q)\sub\dcl_{\cal L(P)}(AB)\sub \dcl_{\cal L(P)}(A_0P_0)=\dcl(A_0P_0).$$ Since $P$ is $\dcl$-independent, there can  be at most $|A_0|$ many such $g$'s, and hence so can $q$'s.
\end{proof}


Two particular cases of the above lemma are the following (recall, $\cal A\sub \dcl(\es)$).

\begin{cor}\label{S0}
Let $C=h\left(\bigcup_{g\in S}\{g\}\times J_g\right)$ be an $H$-cone. Then there is a finite set $S_0\sub S$ such that $h\left(\bigcup_{g\in S\sm S_0}\{g\}\times J_g\right)$ contains no algebraic points.
\end{cor}

\begin{cor}\label{smallA}
Every small set contains only finitely many algebraic points.
\end{cor}
\begin{proof}
By Lemma \ref{S0c}, for $k=0$, and Fact \ref{cp}.
\end{proof}


The key lemma in the inductive step of the proof of Theorem \ref{main-indep} is the following.

\begin{lemma}\label{Fdim0} Let $J\sub \R^k$ be a supercone with shell $Z$, and $B\sub \R$ finite. Then there is an $\cal L$-definable set $F\sub Z$ with $\dim(F)<k$, such that
$$Z\cap \dcl(B)^k\sub J\cup F.$$
\end{lemma}
\begin{proof}
 By induction on $k$. For $k=0$, the statement is trivial. For $k>0$, assume  $J =\bigcup_{g\in \Gamma}\{g\}\times J_g$, where $\Gamma\sub \R^{k-1}$ is a supercone. By inductive hypothesis, there is $F_1\sub \pi(Z)$, such that
 $$\pi(Z)\cap \dcl(B)^{k-1}\sub \Gamma\cup F_1.$$
 Since $\dim (F_1\times \R)<k$, it suffices to write $\left(\bigcup_{g\in \Gamma}\{g\}\times Z_g\right)\cap \dcl(B)^k$ as a subset of $J\cup F_2$, for some $F_2\sub Z$ with $\dim(F_2)<k$. Let  $$X=\bigcup_{g\in \Gamma} \{g\}\times (Z_g\sm J_g).$$
 So we need to prove that $X\cap \dcl(B)^k$ is contained in an $\cal L$-definable set $F_2\sub Z$ with $\dim(F_2)<k$. By \cite[Theorem 2.2]{egh2} and \cite[Corollary 5.11]{egh},
$X$ is a finite union of sets $X_1, \dots, X_l$, each of the form
$$X_i=f\left(\bigcup_{g\in S}\{g\}\times U_g\right),$$
where
\begin{itemize}
  \item $f:V\sub \R^{m+k-1}\to \R^k$ is an $\cal L$-definable continuous map,
  \item $U\sub (S\times \Gamma)\cap V$ is a definable set, and
   \item $f_{\res U}$ is injective.
\end{itemize}
Using Fact \ref{cp}, we may further assume that $S\sub P^m$. By Lemma \ref{S0c}, for $h=f$, there is a finite set $S_0\sub P^m$ such that $$f\left(\bigcup_{g\in S\sm S_0}\{g\}\times U_g\right)\cap \dcl(B)^k=\es.$$
For each $i=1, \dots, l$, and $X_i$ as above, set
 $$D_i=f\left(\bigcup_{g\in S_0}\{g\}\times U_g,\right).$$
Then $F_2=\bigcup_{i=1}^l D_i$ satisfies the required properties.
\end{proof}

\begin{cor}\label{Fdim} Let $C=h(J)\sub \R^n$, where $J\sub \R^k$ is a supercone with shell $Z$, and $h:Z\to \R^n$ an $\cal L$-definable and injective map. Then there is a definable set $F\sub Z$ with $\dim(F)<k$, such that all algebraic points of $h(Z)$ are contained in $h(J\cup F)$.
\end{cor}
\begin{proof} Suppose $h$ is $\cal L_B$-definable, and take $F$ be as in Lemma \ref{Fdim0}. Let $x=h(y)\in h(Z)$ be an algebraic point. In particular, $x\in \dcl(\es)$. Since $h$ is $\cal L$-definable and injective, $y\in \dcl(B)\sub  J\cup F$.
\end{proof}


\begin{thm}\label{main-indep} For every definable set $X$, $X\sm X^{alg_\Q}_t$ has few algebraic points.
\end{thm}
\begin{proof} Let $X\sub \R^n$ be a definable set. We work by induction on the large dimension of $X$. If $\ldim(X)=0$, then $X$ is small and the statement follows from Corollary \ref{smallA}. Assume $\ldim(X)=k>0$. By Facts \ref{fact-cone} and \ref{union}(3), we may assume that $X$ is a $k$-$H$-cone, say $h(\cal J)$ with $\cal J\sub \R^{m+k}$. By Corollary \ref{S0}, we may further assume that $\pi_m(\cal J)$ is a singleton, and hence, that $X=h(J)\sub \R^n$, where $J\sub \R^k$ is a supercone. Let $Z$ be the shell of $J$, and $F\sub Z\sm J$  as in Corollary \ref{Fdim}. We have that $X\sub h(Z\sm F)\cup h(F)$. By Fact \ref{union}(3), it suffices to show the statement for each of $X\cap h(Z\sm F)$ and $X\cap h(F)$. \\

\noindent $X\cap h(F)$. We have
$$\ldim(X\cap h(F))\le \ldim\, h(F)=\dim h(F)<k,$$
and hence we conclude by inductive hypothesis.\\

\noindent $X\cap h(Z\sm F)$.   Observe that
$$h(Z\sm F)^{alg_\Q}\sub (X\cap h(Z\sm F))^{alg_\Q}_t.$$
Indeed, let $T\sub h(Z\sm F)$ be a $\Q$-set. We need to show that $T\sub cl(X\cap T)$. By the conclusion of Corollary \ref{Fdim0}, $T\cap \cal A^n\sub T\cap X$. Since the set of algebraic points $\cal A$ is dense in $Y$, we obtain that
$$T\sub cl(T\cap\cal A^n)\sub cl(T\cap X),$$
as required. Hence, by Fact \ref{pw}, the sets
$$(X\cap h(Z\sm F))\sm (X\cap h(Z\sm F))^{alg_\Q}_t\sub h(Z\sm F)\sm h(Z\sm F)^{alg_\Q}$$
has few algebraic points.
\end{proof}

We now turn to the proof of Theorem \ref{thm-cor}. Note that Theorem \ref{main-indep} implies that if a definable set $X$ contains many algebraic points, then it is dense in an infinite semialgebraic set. However, the last conclusion by itself does not guarantee that $X$ contains an infinite set definable in $\la \overline \R, P\ra$. For example, let $\cal R=\la \overline \R, \textup{exp}\ra$ and $X=e^P$. Then $X$ is definable (in $\la \cal R, P\ra$),  and dense in $\R$. Suppose, towards a contradiction, that it contains an infinite set $Y$ definable in $\la \overline \R, P\ra$.  Then $Y$ must be small in the sense of $\la \overline \R, P\ra$. Indeed, $e^P$ is small in the sense of $\WR$, and smallness is preserved under reducts, by \cite[Corollary 3.12]{egh}. Now, since $Y$ is small in the sense of $\la \overline \R, P\ra$, by \cite{el-small},  there is a semialgebraic $h:\R^n\to \R$ and $S\sub P^n$, such that $h_{\res S}$ is injective and $h(S)=Y\sub e^P$. We leave it to the reader to verify that this statement contradicts the $\dcl$-independence of $P$.

We need two preliminary lemmas. 

\begin{lemma}\label{containS}
Let $J\sub \R^k$ be a supercone. Then there is $b\in \cal A^k$, such that $$(b+P^k)\cap sh(J)\sub J.$$
In particular, $J$ contains an infinite set which is $\es$-definable in $\la \overline \R, P\ra$.
\end{lemma}
\begin{proof}
Denote $Z=sh(J)$. We work by induction on $k$. For $k=0$, $J=P^{0}=\R^{0}=\{0\}$, and the statement holds.
Now let $k>1$. By inductive hypothesis, there is $b_1\in \cal A^{k-1}$, such that
$$(b_1+P^{k-1})\cap \pi(Z)\sub \pi(J).$$
Let $S=(b_1 +P^{k-1})\cap \pi(Z)$. For  every $t\in S$, the set $(Z_t \sm J_t)-P$ is small, and hence $\bigcup_{t\in S} (Z_t \sm J_t)-P$ is also small. By Lemma \ref{smallA},  the last set contains only finitely many algebraic points. So there is
$$b_2\in \cal A\sm \bigcup_{t\in S} ((Z_t\sm J_t) - P).$$
But then for every $p\in P$ and $t\in S$, if $b_2+p\in Z_t$, then $b_2+p\in J_t$. That is, $(b_2+P)\cap Z_t\sub J_t.$ Therefore, for $b=(b_1, b_2)\in \cal A^k$, we have that
$$(b+P)\cap Z\sub J.$$

For the ``in particular" clause, let $B\sub sh(J)$ be any $\es$-definable open box, and $b$ as above.  Then $(b+P^k)\cap B\sub J$ is $\es$-definable in $\la \overline \R, P\ra$. It is also infinite, by density of $P$ in $\R$.
\end{proof}

\begin{question}
 Let $J\sub \R^k$ be a supercone. Does $J$ contain a set  which is $\es$-definable  in $\la \overline \R, P\ra$ and has large dimension $k$?
\end{question}

\begin{lemma}\label{QsetA}
Let $X\sub \R^n$ be a definable set and $T\sub \R^n$ a $\Q$-set, such that $\cal A^n\cap T\sub X$. Then $\ldim (X\cap T)=\dim T$.
\end{lemma}
\begin{proof} Clearly, $\ldim (X\cap T)\le \ldim T= \dim T$. Let $k=\dim T$.
The set $X\cap T$ is a finite union of $H$-cones. By Corollary \ref{S0}, there are finitely many cones $h_i(J_i)$ contained in $X\cap T$ and containing all algebraic points of $X\cap T$. Since $\cal A^n\cap T\sub X$, $\cal A^n\cap T$ is contained in the union of those cones. So
$$T\sub cl(\cal A^n\cap T)\sub \bigcup_i cl(h_i (J_i)),$$ implying that for some $i$, $\dim cl( h_i(J_i))\ge k$. Therefore, some $J_i$ is a supercone in $\R^k$, implying that $\ldim (X\cap T)\ge k$.
\end{proof}

\begin{theorem}
Let $X\sub \R^n$. If $X$ contains many algebraic points, then it contains an infinite set which is $\es$-definable in $\la \overline \R, P\ra$.
\end{theorem}
\begin{proof} The beginning of the proof is similar to that of Theorem \ref{main-indep}, and thus we are brief.  We work by induction on $\ldim(X)=0$. If $\ldim X=0$, then $X$ is small and the statement holds trivially by Corollary \ref{smallA}. For $\ldim X=k>0$, we may assume that $X=h(J)$ is a $k$-cone, with $J\sub \R^k$. Let $Z$ be the shell of $J$, and $F\sub Z\sm J$  as in Corollary \ref{Fdim}. So one of $X\cap h(F)$ and $X\cap h(Z\sm F)$ must contain many algebraic points. If the former one does, then we can conclude by inductive hypothesis. If the latter one does, then by Fact \ref{pw}, there is a $\Q$-cell $T\sub h(Z\sm F)$. By the conclusion of Corollary \ref{smallA}, $\cal A^n\cap T\sub X$.  By Lemma \ref{QsetA}, $\ldim X\cap T=\dim T$. Also,
$$T\sub \cl(\cal A^n \cap T)\sub cl(X\cap T),$$
and hence if follows easily that
$$\dim cl(X\cap T) =\ldim X\cap T.$$
Now, if $T$ is open,  then  $\ldim X\cap T=n$, 
and hence $X\cap T$ contains a supercone in $\R^n$ (by \cite[Theorem 5.7(1)]{egh}). By Lemma \ref{containS}, $X\cap T$ contains an infinite set which is $\es$-definable in $\la \overline \R, P\ra$. Suppose $T=\Gamma(f)$ and let $\pi:\R^n\to \R^k$ be a coordinate projection that is injective on $T$. Then $\ldim \pi(X\cap T)=k$ and hence $\pi(X\cap T)$  contains a supercone in $\R^k$, and thus, by Lemma \ref{containS}, an infinite set $S$ which is $\es$-definable in $\la \overline \R, P\ra$. Then $\Gamma(f_{\res S})$ is contained in $X$ and is as desired.
\end{proof}

We conclude with a remark that goes also beyond the scope of this section.

\begin{remark}
 Let $X\sub \R^n$ and  $P\sub \R$ be as in Theorem \ref{main}. 
 Define
$$X^{alg}_P=\bigcup \{Y\sub X : \text{$Y$  infinite $\es$-definable in $\la \overline \R, P\ra$}\}.$$ 
It is natural to ask whether $X\sm X^{alg}_P$ has few algebraic points. An affirmative answer to this question would strengthen Theorem \ref{main}, and its contrapositive would imply Theorem \ref{thm-cor}. For the case of dense pairs, it is actually not too hard to adjust the proofs of Lemma \ref{dense1} and Theorem \ref{main-dense} and obtain an affirmative answer. For the case of dense independent sets, the question is open, and it is possible that an affirmative answer to Question \ref{QsetA} could be relevant.
\end{remark}

\end{document}